\documentclass[a4paper,12pt]{amsart}
\usepackage[centertags]{amsmath}
\usepackage{amssymb}
\usepackage{amsthm, float}
\usepackage[english]{babel}
\usepackage[latin1]{inputenc}
\usepackage{fancyhdr}
\usepackage[top=2.5cm,bottom=2.5cm,left=2.5cm,right=2.5cm]{geometry}
\usepackage[T1]{fontenc}
\usepackage[dvipsnames]{xcolor}
\usepackage{fancybox}
\usepackage{graphicx}
\usepackage{pifont}
\usepackage{tablists}
\usepackage{multicol}
\usepackage{mathrsfs}
\usepackage{lastpage, hyperref}
\usepackage{etoolbox}

\usepackage{wasysym}

\usepackage[all, cmtip]{xy}

\usepackage{bbm}

\usepackage{hyperref}
\hypersetup{
	colorlinks=true,
	linkcolor=blue,
	filecolor=magenta,      
	urlcolor=cyan,
	citecolor=violet,
	pdftitle={The Hurewicz model structure on simplicial R-modules},
	}

\makeatletter
\@namedef{subjclassname@2020}{%
	\textup{2020} Mathematics Subject Classification}
\makeatother

\numberwithin{equation}{section}

\theoremstyle{plain}
\newtheorem{corollary}[equation]{Corollary}

\theoremstyle{definition} 

\newtheorem{definition}[equation]{Definition}
\newtheorem{lemma}[equation]{Lemma}
\newtheorem{proposition}[equation]{Proposition}
\newtheorem{example}[equation]{Example}
\newtheorem{remark}[equation]{Remark}

\newcommand{\unit}{\mathbbm{1}} 
\newcommand{\Z}{{\mathbb Z}}  
\newcommand{\cA}{{\mathcal A}}  
\newcommand{\cB}{{\mathcal B}}  
\newcommand{\cC}{{\mathcal C}}  
\newcommand{\cP}{{\mathcal P}}  
\newcommand{\cI}{{\mathcal I}}  
\newcommand{\cE}{{\mathcal E}}  
\newcommand{\cV}{{\mathcal V}}  
\newcommand{\Def}[1]{\textbf{\boldmath{#1}}}

\begin{document}
	\title{The Hurewicz model structure on simplicial $R$-modules}															
	\author{Arnaud Ngopnang Ngompe}
	\address{University of Regina, 3737 Wascana Pkwy, Regina, SK S4S 0A2, Canada}
	\email{\url{ann037@uregina.ca}}
	
	\begin{abstract} 
		By a theorem of Christensen and Hovey, the category of non-negatively graded chain complexes has a model structure, called the h-model structure or Hurewicz model structure, where the weak equivalences are the chain homotopy equivalences. The Dold--Kan correspondence induces a model structure on the category of simplicial modules. In this paper, we give a description of the two model categories and some of their properties, notably the fact that both are monoidal.
	\end{abstract}
	
	\keywords{chain complex, homotopy equivalence, model category, monoidal model category, enriched model category, Dold--Kan correspondence}
	
	\subjclass[2020]{Primary 18G31; Secondary 18N40, 55U10}

	\maketitle
	
	\tableofcontents
	
	\section*{Introduction}
	\subsection*{Motivation}
	For a nice bicomplete abelian category $\cA$, two of the well-known standard model structures on the category $\mathrm{Ch}(\cA)$ of unbounded chain complexes are the projective model structure, for which weak equivalences are quasi-isomorphisms and fibrations are degreewise epimorphisms, and the Hurewicz model structure, for which weak equivalences are chain homotopy equivalences and fibrations are degreewise split epimorphisms. May and Ponto \cite[Chapter~18]{may2011more} gave a clear description of these two model structures when $\cA$ is the category of left $R$-modules for a given ring $R$, together with some of their monoidal properties. The Hurewicz model structure on the category of unbounded chain complexes $\mathrm{Ch}(R)$ was introduced by Golasi\'{n}ski and Gromadzki \cite{golasinski1982homotopy}, using work of Kamps \cite{kamps1978note}. It was later constructed using different methods in \cite{cole1999homotopy}, \cite{schwanzl2002strong}, \cite{christensen2002quillen}, \cite[\S\!I]{williamson2013cylindrical}, and \cite{gillespie2015homotopy}. Quillen \cite[\S\!II.4]{quillen2006homotopical} introduced the projective model structure on the category $\mathrm{Ch}_{\geq0}(\cA)$ of non-negatively graded chain complexes, that is with quasi-isomorphisms as weak equivalences and degreewise monomorphisms with degreewise projective cokernel as cofibrations. One of the motivations of this work was to complete this square, that is to give an explicit description of the Hurewicz model structure on $\mathrm{Ch}_{\geq0}(\cA)$. David White \cite{hmodonNCh} pointed out that such a model structure can be obtained using work of Christensen and Hovey \cite{christensen2002quillen}, which was kindly confirmed to us by Dan Christensen. In \cite{christensen2002quillen}, given a bicomplete abelian category $\cA$ with a projective class, Christensen and Hovey construct a model structure on $\mathrm{Ch}(\cA)$ that reflects the homological algebra of the projective class in the sense that it encodes the Ext groups and more general derived functors. A general description of such a model structure on the category $\mathrm{Ch}_{\geq0}(\cA)$ of non-negatively graded chain complexes of objects of $\cA$ under mild conditions is given by \cite[Corollary~6.4]{christensen2002quillen}. 
	\subsection*{Main results}
	In this work, we show that the Hurewicz model structure on $\mathrm{Ch}_{\geq0}(\cA)$ satisfies similar properties, in particular it is monoidal, as its analogue on unbounded chain complexes $\mathrm{Ch}(\cA)$ and the proofs follow mostly from what is done in \cite[Chapter~18]{may2011more}. The mixed model structure is obtained from the projective and the Hurewicz model structures on $\mathrm{Ch}_{\geq 0 }(\cA)$, and \textbf{Proposition~\ref{qhm}} and \textbf{Proposition~\ref{nonqhm}} give the enrichment relations between those three structures. Our main contribution in this paper is showing that the Hurewicz model structure on the category of simplicial $R$-modules $\mathrm{sMod}_R$ is monoidal, see \textbf{Proposition~\ref{sModmonoidalmodel}}. We also characterize the fibrations and cofibrations respectively in terms of homotopy lifting and homotopy extension, see \textbf{Proposition~\ref{sModmodel}}.
	\subsection*{Organization}
	The paper is organized as follows. In the Section~\ref{preliminary}, we recall some background material on Christensen--Hovey model structures, enriched model categories and some properties of the Dold--Kan correspondence. In Section~\ref{hmodel}, we describe the Hurewicz model structure on $\mathrm{Ch}_{\geq0}(R)$ as an instance of the Christensen--Hovey setup and we give some of its monoidal properties. Now that we have the projective (Quillen) and the Hurewicz model structures on $\mathrm{Ch}_{\geq0}(R)$, in Section~\ref{relation}, we give a description of the resulting mixed model structure together with the enrichment relations between these three model structures. In Section~\ref{inducedmodel}, we show that Dold--Kan correspondence transfers the Hurewicz model structure from $\mathrm{Ch}_{\geq0}(R)$ to a monoidal model structure on $\mathrm{sMod}_{R}$. Finally, in Section~\ref{Bousfield}, we explain how the Bousfield model structure on non-negatively graded cochain complexes $\mathrm{Ch}^{\geq0}(\cA)$ is recovered from the Christensen--Hovey setup.
	\subsection*{Acknowledgments}
	 I am thankful to Martin Frankland for his scientific advice and financial contribution to the realization of this work. I would like also to thank Dan Christensen for valuable discussions, and Scott Balchin for helpful discussions on the Bousfield model structure. 
	
	\section{Preliminaries}\label{preliminary}
	
	In this section, we recall some background material on Christensen--Hovey model structures, enriched model categories and some properties of the Dold--Kan correspondence.
	
	\subsection{Christensen--Hovey model structures\nopunct}\text{\ }
	
	We start by reviewing some background material from \cite{christensen2002quillen}.
	
	\begin{definition}
		Let $\cA$ be an abelian category. For an object $P$, a map \mbox{$f : B \to C$} is said to be \Def{$P$-epic} if the induced map $\cA(P, f ) : \cA(P,B) \to \cA(P, C)$ is a surjection of abelian groups. For a collection of objects $\mathcal{P}$, the map $f : B \to C$ is \Def{$\cP$-epic} if it is $P$-epic for all $P\in \cP$.
	\end{definition}
	
	\begin{example} 
		In an abelian category $\cA$, any split epimorphism $f : B \to C$ is \mbox{$\cA$-epic}. Indeed, $f : B \to C$ is a split epimorphism if and only if the induced map 
		\begin{equation*}
			\cA(P, f ) : \cA(P,B) \to \cA(P, C)
		\end{equation*}
		is a surjection of abelian groups for any object $P$.
	\end{example}
	
	\begin{definition}
		For an abelian category $\cA$, a \Def{projective class} on $\cA$ is a collection $\cP$ of objects of $\cA$ and a collection $\cE$ of maps in $\cA$ such that
		\begin{itemize}
			\item[(i)] $\cE$ is precisely the collection of all $\cP$-epic maps;
			\item[(ii)] $\cP$ is precisely the collection of all objects $P$ such that each map in $\cE$ is $P$-epic;
			\item[(iii)] for each object $B$ there is a map $P \to B$ in $\cE$ with $P \in \cP$.
		\end{itemize}
	\end{definition}

	\begin{example}
		Let $\cA$ be an abelian category. Take $\cP$ to be the collection of all objects and $\cE$ to be the collection of all split epimorphisms $B\to C$. Here $(\cP,\cE)$ is a projective class, called the \textbf{trivial projective class}.
	\end{example}
	
	\begin{example}
		Let us consider a functor $U : \cA\to \cB$ of abelian categories, together with a left adjoint $F : \cB\to \cA$ . Then $U$ and $F$ are additive, $U$ is left exact and $F$ is right exact. If $(\cP', \cE')$ is a projective class on $\cB$, we define 
					\begin{equation*}
					 			\cP := \{\text{retracts of}\ F(P)\ \text{for}\ P \in \cP'\}\ \text{and}\ \cE := \{B \to C\ \text{such that}\ U(B) \to U(C) \in \cE'\}. 
					\end{equation*}
					Then $(\cP, \cE)$ is a projective class on $\cA$, called the \textbf{pullback} of $(\cP', \cE')$ along the right adjoint $U$.
\end{example}
	
	\begin{proposition}\label{cor6.4}
		If $\cA$ is a bicomplete abelian category with a projective class $\cP$, then the category $\mathrm{Ch}_{\geq 0}(\cA)$ of non-negatively graded chain complexes of objects in $\cA$ is endowed with a model structure as follows. 
		\begin{enumerate}
			\item A map $f$ is a weak equivalence if $\cA(P, f)$ is a quasi-isomorphism for each $P \in \cP$. 
			\item A map $f$ is a fibration if $\cA(P, f)$ is surjective in positive degrees (but not necessarily in degree $0$) for each $P \in \cP$. 
			\item A map $f$ is a cofibration if it is a degreewise split monomorphism with degreewise $\cP$-projective cokernel. 
		\end{enumerate}
		Every complex is fibrant, and a complex is cofibrant if and only if it is a complex of $\cP$-projectives.
	\end{proposition}
	
	\subsection{Enriched model categories\nopunct}\text{\ }
	
	Most of the material we review here can be found in \cite[\S4]{hovey1999model} and \linebreak \mbox{\cite[Chapter~16]{may2011more}}.
	
	\begin{definition} A \Def{monoidal model category} is a model category $\cC$ equipped with  the structure of a closed symmetric monoidal category $(\cC,\otimes,\unit)$ such that the  following two compatibility conditions are satisfied.
		\begin{itemize}
			\item[(i)] \textbf{Pushout-product axiom:} For every pair of cofibrations $i:X\to Y$ and \linebreak 
			$k:V\to W$, their pushout-product $i\square k$, that is the induced morphism out of the pushout
			\begin{equation*}
				\xymatrix{
					(X\otimes W)\amalg_{X\otimes V} (Y\otimes V) \ar[r]^-{i\Square k} & Y\otimes W}
			\end{equation*}
			is itself a cofibration, which moreover is acyclic if $i$ or $k$ is.
			\item[(ii)] \textbf{Unit axiom:} For every cofibrant object $X$ and every cofibrant replacement of the tensor unit $q:Q\unit \to \unit$ (i.e. a weak equivalence with cofibrant source) in $\cC$, the resulting morphism 
			\begin{equation*}
				\xymatrix{X\otimes Q\unit \ar[r]^-{X\otimes q} & X\otimes \unit \ar[r]^-{\cong} & X} 
			\end{equation*}
			is a weak equivalence.
		\end{itemize}
	\end{definition}
	
	\begin{definition}\label{def-emc} Let $\cV$ be a monoidal model category.
		A \Def{\mbox{$\cV$-enriched} model category} is a $\cV$-enriched category $\underline{\cC}$, which is tensored and cotensored over $\cV$, with the structure of a model category on the underlying category $\cC$ such that the  following two compatibility conditions are satisfied.
		\begin{itemize}
			\item[(i)] \textbf{(External) Pushout-product axiom:} For every pair of cofibrations $i:X\to Y$ in $\cC$ and $k:V\to W$ in $\cV$, their pushout-product $i\square k$
			\begin{equation*}
				\xymatrix{
					(X\otimes W)\amalg_{X\otimes V} (Y\otimes V) \ar[r]^-{i\Square k} & Y\otimes W} 
			\end{equation*}
			is itself a cofibration in $\cC$, which moreover is acyclic if $i$ or $k$ is.
			\item[(ii)] \textbf{Unit axiom:} For every cofibrant object $X$ in $\cC$ and every cofibrant replacement of the tensor unit $q:Q\unit \to \unit$ in $\cV$, the resulting morphism \begin{equation*}
				\xymatrix{X\otimes Q\unit \ar[r]^-{X\otimes q} & X\otimes \unit \ar[r]^-{\cong} & X} 
			\end{equation*}
			is a weak equivalence in $\cC$.
		\end{itemize}
\end{definition}
	
	\subsection{Chain complexes and Dold--Kan correspondence\nopunct}\text{\ }
	
	Background material on tensored and cotensored categories can be found in \linebreak \cite[\S1]{kelly1982basic}, \cite[\S 6.5]{borceux1994handbook}  and \cite[\S 3.7]{riehl2014categorical}.
	
	We recall the definition of the tensor product and the hom complex of chain complexes of $R$-modules. Before, we recall the following about $R$-modules.
	\begin{itemize}
		\item For a commutative ring $R$, the category of $R$-modules $\mathrm{Mod}_R$ endowed with the usual tensor product over $R$ and  the $R$-module structure on its hom set, \linebreak \mbox{$(\mathrm{Mod}_R, \otimes_R, R, \mathrm{Hom}_R)$} is a closed symmetric monoidal category.
		\item  For an arbitrary ring $R$, the category of left $R$-modules $\mathrm{Mod}_R$ is enriched, tensored, and cotensored over abelian groups $\mathrm{Ab}$. Here, for an abelian group $A$ and a left $R$-module $M$, the action of $R$ on the tensoring $M\otimes_{\Z} A$ is given by:
		\begin{equation*}
			r(m\otimes a) := (rm)\otimes a,\text{ for all }r\in R, m\in M\ \text{and}\ a\in A.
		\end{equation*} 
		Similarly, the action of $R$ on the cotensoring $\mathrm{Hom}_{\Z}(A,M)$ is given by:
		\begin{equation*}
			(rf)(a):=r(f(a)),\text{ for all } r\in R, f\in \mathrm{Hom}_{\Z}(A,M)\ \text{and}\ a\in A.
		\end{equation*}
	\end{itemize}
		
		\begin{definition}\label{tenCh} For a commutative ring $R$, the category of chain complexes of $R$-modules $\mathrm{Ch}(R)$ is endowed with a \Def{tensor product} defined as follows:
				\begin{equation*}
					 	(X\otimes Y)_n := \bigoplus_{i+j=n}X_i\otimes_R Y_j,\quad \text{with}\quad d(x\otimes y) :=  d(x)\otimes y+(-1)^{|x|}x\otimes d(y).
				\end{equation*}
		\end{definition}
				
		\begin{definition}\label{homCh} For an arbitrary ring $R$, the category of chain complexes of $R$-modules $\mathrm{Ch}(R)$ is endowed with a \Def{hom complex} defined as follows:
				\begin{equation*}
					 	\mathrm{\underline{Hom}}_{\mathrm{Ch}(R)}(X,Y)_n := \prod_{i\in \Z} \mathrm{Hom}_R(X_i,Y_{i+n}), \quad \text{with}\quad (df)(x) := d(f(x))-(-1)^{|f|}f(d(x))
				\end{equation*}
				which is a chain complex of abelian groups. If $R$ is commutative, then the hom complex is a chain complex of $R$-modules.
		\end{definition}
	
	    \begin{lemma}\label{unboundmonoidal}
	    	For a commutative ring $R$, the category of unbounded chain complexes $\mathrm{Ch}(R)$ endowed with the tensor product $\otimes$ and the hom complex $\mathrm{\underline{Hom}}_{\mathrm{Ch}(R)}$, \linebreak
	    	$(\mathrm{Ch}(R), \otimes, R[0], \mathrm{\underline{Hom}}_{\mathrm{Ch}(R)})$ is a closed symmetric monoidal category. Here the tensor unit $R[0]$ is the chain complex with $R$ concentrated in degree $0$.
	    \end{lemma}
		
		\begin{definition}
			The \Def{good truncation functor} 
			\begin{equation*}
			\tau_{\geq 0}: \mathrm{Ch}(R) \to \mathrm{Ch}_{\geq 0}(R)
		\end{equation*}
		is defined by 
		\begin{equation*}
			(\tau_{\geq 0}C)_n:=\begin{cases}
				C_n,\ &\text{if}\ n\geq 1 \\
				\mathrm{ker}(d_0),\ &\text{if}\ n=0 \\
				0,\ &\text{otherwise}
			\end{cases}
		\end{equation*}
	for any complex $C$ in $\mathrm{Ch}(R)$.
		\end{definition}
	
	\begin{lemma}\label{geq0monoidal}
		\begin{enumerate}
			\item For a commutative ring $R$, the category $\mathrm{Ch}_{\geq 0}(R)$ of non-negatively graded chain complexes of $R$-modules endowed with the tensor product $\otimes$ and the hom complex 
			\begin{equation*}
				 			\mathrm{\underline{Hom}}_{\mathrm{Ch}_{\geq 0}(R)} := \tau_{\geq 0}(\mathrm{\underline{Hom}}_{\mathrm{Ch}(R)}), 
			\end{equation*}
			$(\mathrm{Ch}_{\geq 0}(R), \otimes, R[0], \mathrm{\underline{Hom}}_{\mathrm{Ch}_{\geq 0}(R)})$ is a closed symmetric monoidal category.
			\item For an arbitrary ring $R$, the category $\mathrm{Ch}_{\geq 0}(R)$ of non-negatively graded chain complexes of $R$-modules is enriched, tensored and cotensored over the category $\mathrm{Ch}_{\geq 0}(\Z)$ of non-negatively graded chain complexes of abelian groups.
			\begin{itemize}
				\item The enrichment is given by $\mathrm{\underline{Hom}}_{\mathrm{Ch}_{\geq 0}(R)}(X,Y)$ in $\mathrm{Ch}_{\geq 0}(\Z)$ as defined in item (1), for any objects $X$ and $Y$ in $\mathrm{Ch}_{\geq 0}(R)$.
				\item The tensoring is given by $X\otimes K$ in $\mathrm{Ch}_{\geq 0}(R)$ as defined in \textbf{Definition \ref{tenCh}} with $R=\Z$, for any objects $X$ in $\mathrm{Ch}_{\geq 0}(R)$ and $K$ in $\mathrm{Ch}_{\geq 0}(\Z)$.
				\item The cotensoring is given by $X^K:=\mathrm{\underline{Hom}}_{\mathrm{Ch}_{\geq 0}(\Z)}(K,X)$ in $\mathrm{Ch}_{\geq 0}(R)$ as defined in item (1) with $R=\Z$, for any objects $X$ in $\mathrm{Ch}_{\geq 0}(R)$ and $K$ in $\mathrm{Ch}_{\geq 0}(\Z)$.
			\end{itemize}
		\end{enumerate}
	\end{lemma}

The following can be found in \cite[\S\!II.2]{goerssJard2009simplicial}.
	
	\begin{lemma}\label{sModmonoidal}
		\begin{enumerate}
			\item For a commutative ring $R$, the category $\mathrm{sMod}_R$ of simplicial $R$-modules has a closed symmetric monoidal structure. The monoidal product is given by the degreewise tensor product defined as follows: 
			\begin{equation*}
				(A\otimes B)_n := A_n\otimes_R B_n,
			\end{equation*}
			for any simplicial $R$-modules $A$ and $B$. The constant simplicial $R$-module $c(R)$ is the tensor unit, that is:
			\begin{equation*}
				c(R)\otimes A \cong A.
			\end{equation*}
			The internal hom of simplicial $R$-modules $A$ and $B$ is the simplicial $R$-module $\mathrm{\underline{Hom}}_{\mathrm{sAb}}(A, B)$ given in degree $n$ by the $R$-module
			\begin{equation*}
				\mathrm{\underline{Hom}}_{\mathrm{sMod}_R}(A, B)_n := \mathrm{Hom}_{\mathrm{sMod}_R}(A\otimes R\Delta^n, B).
			\end{equation*}
			\item For an arbitrary ring $R$, the category $\mathrm{sMod}_R$ of simplicial $R$-modules is enriched, tensored and cotensored over the category $\mathrm{sAb}$ of simplicial abelian groups.
			\begin{itemize}
				\item The tensoring is given by $A\otimes K$ in $\mathrm{sMod}_R$ as defined in item (1) with $R=\Z$, for any objects $A$ in $\mathrm{sMod}_R$ and $K$ in $\mathrm{sAb}$.
				\item The enrichment is given by: 
				\begin{equation*}
					\mathrm{\underline{Hom}}_{\mathrm{sMod}_R}(A,B) := \mathrm{Hom}_{\mathrm{sMod}_R}(A\otimes \Z\Delta^n, B)
				\end{equation*}
				in $\mathrm{sAb}$, for any objects $A$ and $B$ in $\mathrm{sMod}_R$.
				\item The cotensoring is given by $A^K:=\mathrm{\underline{Hom}}_{\mathrm{sAb}}(K,A)$ in $\mathrm{sMod}_R$ as defined in item (1) with $R=\Z$, for any objects $A$ in $\mathrm{sMod}_R$ and $K$ in $\mathrm{sAb}$.
			\end{itemize}
		\end{enumerate}
	\end{lemma}
	
	The Dold--Kan correspondence is given by the two functors
		\begin{equation*}
		\xymatrix{
			N: \mathrm{sMod}_R  \ar@<1.8ex>[r] & \ar@<.6ex>[l]_\cong \mathrm{Ch}_{\geq 0}(R): \Gamma,
		}
	\end{equation*}
	where $N$ and $\Gamma$ are respectively called the \Def{normalization} and the \Def{denormalization}.
	
	The following can be found in \cite[\S2.3]{schwede2003equivalences}, where the map denoted $\mathrm{EZ}$ is the \textbf{Eilenberg--Zilber map}, also known as the \textbf{shuffle map}, and the map denoted $\mathrm{AW}$ is the \textbf{Alexander--Whitney map}.
	\begin{proposition}\label{EZAWProp} 
			For simplicial $R$-modules $A$ and $B$, the chain complex $N(A)\otimes N(B)$ is a deformation retract of $N(A\otimes B)$ given by the following two diagrams
			
			\begin{center}
				\begin{tabular}{cc}
				$ \tiny
				\xymatrix{
					N(A)\otimes N(B) \ar[r]^-{\mathrm{EZ}}  \ar@/_3.2pc/[rr]_-{\mathrm{id}}
					& N(A\otimes B) \ar[r]^-{\mathrm{AW}} \ar@{=}[d] & N(A)\otimes N(B), \\
					& & 
				}
			$
			&
			$\tiny
				\xymatrix{
					N(A\otimes B) \ar[r]^-{\mathrm{AW}}  \ar@/_3.2pc/[rr]_-{\mathrm{id}}
					& N(A)\otimes N(B) \ar[r]^-{\mathrm{EZ}} \ar@{=>}[d] & N(A\otimes B). \\
					& & 
				}$
			\end{tabular}
			\end{center}
			That is, the composite $\mathrm{AW} \circ \mathrm{EZ}$ is the identity, and the composite $\mathrm{EZ} \circ \mathrm{AW}$ is naturally chain homotopic to the identity.
	\end{proposition}
	
	\textbf{Lemmas \ref{unboundmonoidal}}, \textbf{\ref{geq0monoidal}} and \textbf{\ref{sModmonoidal}} can be extended to the case where the category of $R$-modules $\mathrm{Mod}_R$ is replaced by an abelian category $\cA$ as given by the following two propositions.
	
	\begin{proposition}\label{Prop1}
		For a bicomplete closed symmetric monoidal abelian category $\cA$,
		\begin{enumerate}
			\item $\mathrm{Ch}_{\geq 0}(\cA)$ is a closed symmetric monoidal category,
			\item $\mathrm{Ch}(\cA)$ is a closed symmetric monoidal category,
			\item $s\cA$ is a closed symmetric monoidal category.
		\end{enumerate}
	\end{proposition}
	
	\begin{proposition}\label{Prop2}
		For a bicomplete abelian category $\cA$,
		\begin{enumerate}
			\item $\mathrm{Ch}_{\geq 0}(\cA)$ is enriched, tensored and cotensored over $\mathrm{Ch}_{\geq 0}(\Z)$,
			\item $\mathrm{Ch}(\cA)$ is enriched, tensored and cotensored over $\mathrm{Ch}(\Z)$,
			\item $s\cA$ is enriched, tensored and cotensored over $\mathrm{sAb}$.
		\end{enumerate}
	\end{proposition}
	
	\begin{remark}
		All the results in Sections \ref{hmodel}, \ref{relation} and \ref{inducedmodel} are true for an abelian category $\cA$ satisfying the assumptions of \textbf{Proposition \ref{Prop1}} or \textbf{\ref{Prop2}}. But mostly we focus on the case $\cA=\mathrm{Mod}_R$ for convenience.
	\end{remark}
	
	\section{Hurewicz model structure on \texorpdfstring{$\mathrm{Ch}_{\geq 0}(\cA)$}{Ch0(A)}}\label{hmodel}
	
	In this section, we show that the Hurewicz model structure on $\mathrm{Ch}_{\geq0}(\cA)$ is an instance of the Christensen--Hovey model structures, for a bicomplete abelian category $\cA$. Also we prove some related properties.
	
	\begin{lemma}\label{Yoneda game}
		Let $\cB$ be an abelian category and $g:X\to Y$ a map in $\mathrm{Ch}(\cB)$. The map $g$ is a degreewise split epimorphism if and only if for every $M\in \cB$, the map \mbox{$g_*:\cB(M,X)\to \cB(M,Y)$} in $\mathrm{Ch}(\Z)$ is a degreewise surjection.
	\end{lemma}
	
	\begin{proof}
		$(\implies)$ Verified since split epimorphisms are universal, hence the image of the degreewise split epimorphism $g$ by the functor $\cB(M,-):\mathrm{Ch}(\cB)\to \mathrm{Ch}(\Z)$ is a degreewise split epimorphism. 
		
		\noindent $(\impliedby)$ For $M=Y_n$, we have
			\begin{center}
				\begin{tabular}{ccc}
					$\xymatrix{\cdots \ar[d] &  \cdots \ar[d] \\
						X_{n+1} \ar[d]_-{\partial} \ar[r]^-{g_{n+1}} & Y_{n+1} \ar[d]^-{\partial} \\
						X_{n} \ar[d]_-{\partial} \ar[r]^-{g_{n}} & Y_{n} \ar@{.>}@/^/[l]^-{s_n?} \ar[d]^-{\partial} \\
						X_{n-1} \ar[d] \ar[r]^-{g_{n-1}} & Y_{n-1} \ar[d] \\
						\cdots &  \cdots }$
					
					& $\xymatrix{ & \\
						& \\
						& \\
						& \implies\\
						& \\
						& }$ &
					
					$\xymatrix{\cdots \ar[d] &  \cdots \ar[d] \\
						\cB(Y_n,X_{n+1}) \ar[d]_-{\partial} \ar[r]^-{(g_{n+1})_*} & \cB(Y_n,Y_{n+1}) \ar[d]^-{\partial} \\\
						\cB(Y_n,X_{n}) \ar[d]_-{\partial} \ar[r]^-{(g_{n})_*} & \cB(Y_n,Y_{n}) \ar[d]^-{\partial} \\
						\cB(Y_n,X_{n-1}) \ar[d] \ar[r]^-{(g_{n-1})_*} & \cB(Y_n,Y_{n-1}) \ar[d] \\
						\cdots &  \cdots }$
				\end{tabular}
			\end{center}

		Then take $s_n\in \cB(Y_n,X_n)$ such that $(g_n)_*(s_n) = g_ns_n = \mathrm{id}_{Y_n}$.
	\end{proof}
	
	\begin{lemma}\label{lem3.1}
		If $U : \cA \to \cB$ is a functor of abelian categories, with a left adjoint \mbox{$F : \cB \to \cA$}. Let $\cP$ be the projective class on $\cA$ that is the pullback of the trivial projective class on $\cB$, then
		\begin{enumerate}
			\item A map $f : X \to Y$ in $\mathrm{Ch}_{\geq0}(\cA)$ is a $\cP$-equivalence if and only if the map \mbox{$Uf : UX \to UY$} is a chain homotopy equivalence in $\mathrm{Ch}_{\geq0}(\cB)$.
			\item A map $f : X \to Y$ in $\mathrm{Ch}_{\geq0}(\cA)$ is a $\cP$-fibration if and only if the map \linebreak
			$Uf : UX \to UY$ in $\mathrm{Ch}_{\geq0}(\cB)$ is a degreewise split epimorphism in positive degrees.
		\end{enumerate}
	\end{lemma}
	
	\begin{proof}
		(1) The proof of \cite[Lemma~3.1~(b)]{christensen2002quillen} is still valid here since the cofiber $C$ of $Uf$ is given by $C_n=(UY)_n\oplus (UX)_{n-1}$ and so $C\in \mathrm{Ch}_{\geq0}(\cB)$. 
		
		\noindent (2) By \textbf{Proposition~\ref{cor6.4}}, $f : X \to Y$ in $\mathrm{Ch}_{\geq0}(\cA)$ is a $\cP$-fibration if and only if the bijection given by the adjunction $F \dashv U$
			\[\cA(FM, f)\cong\cB(M, Uf):\cB(M, UX)\to\cB(M, UY)\]
			is degreewise surjective in positive degrees for every $M \in \cB$ (trivial projective class). By \textbf{Lemma~\ref{Yoneda game}}, this is equivalent to $Uf:UX:\to UY$ being degreewise split epimorphism in positive degrees. 
	\end{proof}
	
	\begin{proposition}\label{h-mod}
		For $\cA$ a bicomplete abelian category, $\mathrm{Ch}_{\geq 0}(\cA)$ has a model structure given by the following.
		\begin{enumerate}
			\item A map $f$ is an weak equivalence if it is a chain homotopy equivalence.
			\item  A map $f$ is an fibration if it is a degreewise split epimorphism in positive degrees (not necessarily in degree zero).
			\item  A map $f$ is an cofibration if it is a degreewise split monomorphism.
		\end{enumerate}
		Moreover, every complex is fibrant and cofibrant.
	\end{proposition}
	
	\begin{proof}
		Parts~(1) and (2) are obtained by applying \textbf{Lemma~\ref{lem3.1}} with $U=\mathrm{id}_{\cA}:\cA\to \cA$. In other words, take the trivial projective class on $\cA$. 
		
		\noindent Part~(3) is given by \textbf{Proposition~\ref{cor6.4}} and so a map $f:X\to Y$ in $\mathrm{Ch}_{\geq 0}(\cA)$ is an h-cofibration if and only  if it is a degreewise split monomorphism, and every complex is fibrant and cofibrant, since all objects in $\cA$ are $\cP$-projectives.
	\end{proof}
	
	The model structure on $\mathrm{Ch}_{\geq 0}(\cA)$ defined in \textbf{Proposition \ref{h-mod}} is called the \linebreak
	\Def{Hurewicz model structure}, in light of the following characterization of the cofibrations and fibrations.
		
	\begin{proposition}\text{\ }
		\begin{enumerate}
			\item A map $i : A \to X$ in $\mathrm{Ch}_{\geq 0}(\cA)$ is a cofibration if and only if it satisfies the homotopy extension property (HEP).
			\item A map $p : E \to B$ in $\mathrm{Ch}_{\geq 0}(\cA)$ is a fibration if and only if it satisfies the homotopy lifting property (HLP).
		\end{enumerate}
	\end{proposition}
	
	\begin{proof} (1) The proof of \cite[Proposition~18.3.6]{may2011more} applies here, since any degreewise split monomorphism $i:A\to X$ of non-negatively graded chain complexes is an r-cofibration when seen as a map of unbounded complexes and its mapping cylinder $Mi=X\cup_i A\otimes I$, computed in $\mathrm{Ch}(\cA)$, is a non-negatively graded chain complex. 
		
	\noindent (2) $(\implies)$ Assume that a map $p : E \to B$ in $\mathrm{Ch}_{\geq 0}(\cA)$ is a fibration. As in \linebreak
	\mbox{\cite[Proposition~18.3.6]{may2011more}}, since $i_0:A\to A\otimes I$ is an acyclic cofibration, by  \textbf{Proposition~\ref{h-mod}}, the lifting property of the model structure gives that $p$ satisfies the HLP. 
	
	\noindent $(\impliedby)$ Assume that a map $p : E \to B$ in $\mathrm{Ch}_{\geq 0}(\cA)$ satisfies the HLP. Then the map \linebreak 
	$(\mathrm{ev}_0,p_*):E^I\to \tau_{\geq 0}(Np)$ has a section $\sigma:\tau_{\geq 0}(Np)\to E^I$ 
				between the mapping cocylinder $\tau_{\geq 0}(Np)=\tau_{\geq 0}(E\times_B B^I)$ and the path space  $E^I$, where $\tau_{\geq 0}: \mathrm{Ch}(\cA)\to \mathrm{Ch}_{\geq 0}(\cA)$ is the good truncation functor. The map $(\mathrm{ev}_0,p_*)$ at any degree $n> 0$ is given by
				\begin{equation*}
					\xymatrix@C+3pc{ 
						(E^I)_n  \ar[r]^-{(\mathrm{ev}_0,p_*)} \ar[d]^-{\cong} & (\tau_{\geq 0}(E\times_B B^I))_n \ar[d]^-{\cong} \\
						E_n\oplus E_n\oplus E_{n+1} \ar[r]^-{(\mathrm{Id},p_n,p_{n+1})} &  E_n\oplus B_n\oplus B_{n+1}. \vspace{.01cm}
					} 
				\end{equation*}
				For $n > 0$, the map $s_n : B_n\to E_n$ given by the composition
				\begin{equation*}
					\xymatrix{ 
						B_n  \ar@{^{(}->}[r] \ar@/^3pc/[rrr]_-{s_n} & (E\times_B B^I)_n \ar[r]^-{\sigma_n} \ar[d]^-{\cong} & (E^I)_n \ar[r]^-{(\mathrm{ev}_1)_n} \ar[d]^-{\cong} & E_n \\
						B_n \ar@{^{(}->}[r] & E_n\oplus B_n\oplus B_{n+1} \ar[r]^-{\sigma_n} & E_n\oplus E_n\oplus E_{n+1} \ar[r]^-{\mathrm{proj}_2} & E_n \\
						b \ar@{|->}[r] & (0,b,0) \ar@{|->}[r] & (e_1,e_2,e_3) \ar@{|->}[r] & e_2
					} 
				\end{equation*}
			\vspace{-0.1cm}
				is a section of $p_n : E_n \to B_n$. Indeed, for any $(e,b_1,b_2)\in E_n\oplus B_n\oplus B_{n+1}$,
				\vspace{-0.1cm}
				\begin{align}\label{sec}
					\small (e,b_1,b_2) = (\mathrm{ev}_0,p_*)_n\sigma_n(e,b_1,b_2) = (\mathrm{ev}_0,p_*)_n(e_1,e_2,e_3) & = (e_1,p_n(e_2),p_{n+1}(e_3)) \nonumber\\
					\implies p_n(e_2) & = b_1
				\end{align}
			\vspace{-0.1cm}
				and so for any $b\in B_n$,
				\vspace{-0.1cm}
				\[p_ns_n(b)=p_n\mathrm{proj}_2\sigma_n(0,b,0)=p_n\mathrm{proj}_2(e_1,e_2,e_3)=p_n(e_2)=b,\ \text{by}\ (\ref{sec}).  \qedhere\]
	\end{proof}
	
	\begin{example}
		The brutal truncation $q:C\to C/C_0$ is a fibration in the model category defined in \textbf{Proposition \ref{h-mod}}, hence it satisfies the HLP.  Indeed, considering the right hand side diagram with $f\simeq g$ with the homotopy $H$, there is the lift $\widetilde{H}$ of $H$ such that \linebreak
		$\widetilde{f}\simeq \widetilde{g}$ with the homotopy $\widetilde{H}$.
		
		\begin{center}
			\begin{tabular}{ccc}
				$\xymatrix{
					&& \\
					&& \\
					&& C \ar[dd]^-{q}\\
					&&\\
					A 
					\ar[uurr]^-{\widetilde{f}}
					\ar@/^1.5pc/[rr]_{\quad}^{f}="1" 
					\ar@/_1.5pc/[rr]_{g}="2" 
					&&  C/C_0
					\ar@{}"1";"2"|(0.135){\,}="7" 
					\ar@{}"1";"2"|(0.875){\,}="8" 
					\ar@{=>}|-{H}"7" ;"8"
				} $  & $\xymatrix{
					& \\
					& \\
					& \\
					& \implies
				}$  &
				$\xymatrix{\cdots \ar[d] & \cdots \ar[d] & \cdots \ar[d] \\
					C_2 \ar[d]
					A_2 \ar[ur]^-{\widetilde{H}_2} \ar[d] \ar[r]^-{\widetilde{f}_2,\widetilde{g}_2} & C_2 \ar[d] \ar[r]^q & C_2 \ar[d] \\
					A_1 \ar[ur]^-{\widetilde{H}_1}\ar[d] \ar[r]^-{\widetilde{f}_1,\widetilde{g}_1} & C_1 \ar[d] \ar[r]^q & C_1 \ar[d] \\
					A_0 \ar[ur]^-{\widetilde{H}_0}\ar[d] \ar[r]^-{\widetilde{f}_0,\widetilde{g}_0} & C_0 \ar[d] \ar[r] & 0 \ar[d] \\
					0  \ar[ur]^-{\widetilde{H}_{-1}} & 0 & 0 
				}$
			\end{tabular}
		\end{center}
			
		We want to lift the homotopy $H$ and obtain $\partial \widetilde{H}+\widetilde{H}\partial = \widetilde{g}-\widetilde{f}$. For $n\geq 1$, we have $\widetilde{H}_n=H_n$ and $\widetilde{g}_n=g_n$. For $n=0$, we have
				$\partial \widetilde{H}_0+\widetilde{H}_{-1}\partial  = \widetilde{g}_0-\widetilde{f}_0 $. So
				$\widetilde{g}_0  = \widetilde{f}_0 + \partial H_0$, with $\widetilde{H}_0 = H_0$ since $\widetilde{H}_{-1}=0$.
	\end{example}
	
	\begin{proposition}\label{monoidalprop}\text{\ }
	\begin{enumerate}
		\item For a commutative ring $R$, the Hurewicz model structure on $\mathrm{Ch}_{\geq 0}(R)$ is monoidal.
		\item For an arbitrary ring $R$, the Hurewicz model structure on $\mathrm{Ch}_{\geq 0}(R)$ is enriched over $\mathrm{Ch}_{\geq 0}(\Z)$.
	\end{enumerate}
	\end{proposition}

	\begin{proof}
		The proof follows from its analogue on $\mathrm{Ch}(R)$ \cite[End of \S18.3]{may2011more}. If we take maps $i:A\to X$ in $\mathrm{Ch}_{\geq 0}(R)$ and $j:Y\to Z$ in $\mathrm{Ch}_{\geq 0}(R)$ (or in $\mathrm{Ch}_{\geq 0}(\Z)$), then all objects in the diagrams in that proof are non-negatively graded, that is, in $\mathrm{Ch}_{\geq 0}(R)$ or $\mathrm{Ch}_{\geq 0}(\Z)$.
	\end{proof}

\section{Relation with other model structures on \texorpdfstring{$\mathrm{Ch}_{\geq 0}(R)$}{Ch0(R)}}\label{relation}
	
	Given the Hurewicz model structure described in the previous section, by \linebreak \mbox{\cite[Theorem~2.1]{cole2006mixing}} there is a new model structure on non-negatively graded chain complexes $\mathrm{Ch}_{\geq0}(R)$ given by the following proposition.
	
	\begin{proposition}
		There is a \textbf{mixed model structure} on $\mathrm{Ch}_{\geq 0}(R)$ with
		\begin{enumerate}
			\item m-weak equivalences: quasi-isomorphisms,
			\item m-fibrations: degreewise split epimorphisms in positive degrees (not necessarily in degree zero),
			\item m-cofibrations: maps satisfying the left lifting property (LLP) with respect to acyclic fibrations ($\textbf{F}\cap\textbf{W}$).
		\end{enumerate}
	\end{proposition}
	
	\begin{remark}\label{mmonoidal}
		Since the Quillen and Hurewicz model categories, respectively denoted by $\mathrm{Ch}_{\geq 0}(R)_q$ \cite[Remark~1.6]{jardine2003presheaves} and $\mathrm{Ch}_{\geq 0}(R)_h$ are monoidal, then the above mentioned mixed model category, denoted by $\mathrm{Ch}_{\geq 0}(R)_m$ is also monoidal, by \mbox{\cite[Proposition~6.6]{cole2006mixing}}.
	\end{remark}
	
	\begin{lemma}\label{q2h-we}
		In $\mathrm{Ch}_{\geq 0}(R)$ (or in $\mathrm{Ch}(R)$), if a map $f$ is a q-cofibration and a q-weak equivalence, i.e. a quasi-isomorphism, then $f$ is a h-weak equivalence, i.e. a chain homotopy equivalence.
	\end{lemma}
	
	\begin{proof}
		Let $\xymatrix{ f:A\ar@{^{(}->}[r]^-{\sim} & X}$ be an acyclic q-cofibration. We have the following split short exact sequence of graded $R$-modules
		\begin{equation*}
			\xymatrix{ & & & \text{coker}(f)\ar@{=}[d] & \\
				0\ar[r] & A\ar[r]^-{\sim} & X\ar[r] & X/A \ar[r] & 0}
		\end{equation*}
		since $\text{coker}(f)=X/A$ is degreewise projective. Since $f$ is a quasi-isomorphism, we have $H_*(X/A)=0$ by the long exact sequence, i.e. $X/A$ is weakly contractible. Also, $X/A$ is a fibrant and cofibrant chain complex, and so it is contractible by Whitehead's Theorem in model categories \cite[Proposition~1.2.8]{hovey1999model}. Therefore, by \cite[Lemma~18.2.8]{may2011more}, $f$~is a chain homotopy equivalence.
	\end{proof}
	
	\begin{lemma}\label{m2h-we}
		In $\mathrm{Ch}_{\geq 0}(R)$ (or in $\mathrm{Ch}(R)$), if a map $j$ is a m-cofibration and an m-weak equivalence, i.e.\ a quasi-isomorphism, then $j$ is an h-weak equivalence, i.e.\ a chain homotopy equivalence.
	\end{lemma}
	
	\begin{proof}
		Let $\xymatrix{j:A\ar@{^{(}->}[r]^-{\sim} & X}$ be an acyclic m-cofibration, i.e. an m-cofibration and a quasi-isomorphism. As an m-cofibration, by \cite[Theorem~17.3.5]{may2011more}, $j$ factors as follows
		\begin{equation*}
			\xymatrix{ A\ar@{^{(}->}[rd]_-i \ar@{^{(}->}[rr]_-{\simeq}^-j & & X \\
				& X' \ar[ru]^-{\simeq}_-f & },
		\end{equation*}
		where $i$ is a q-cofibration and $f$~is a chain homotopy equivalence. Since $j$~is also a quasi-isomorphism, then by 2-out-of-3, $i$~is a quasi-isomorphism too. Hence by \textbf{Lemma~\ref{q2h-we}}, $i$~is a chain homotopy equivalence and so by 2-out-of-3, $j$~is a chain homotopy equivalence too.
	\end{proof}
	
	We have the following relations between three of the model structures on $\mathrm{Ch}_{\geq 0}(R)$.
	\begin{proposition}\label{qhm}
		For a commutative ring $R$,
		\begin{enumerate}
			\item $\mathrm{Ch}_{\geq 0}(R)_h$ is enriched over $\mathrm{Ch}_{\geq 0}(R)_q$,
			\item  $\mathrm{Ch}_{\geq 0}(R)_h$ is enriched over $\mathrm{Ch}_{\geq 0}(R)_m$,
			\item $\mathrm{Ch}_{\geq 0}(R)_m$ is enriched over $\mathrm{Ch}_{\geq 0}(R)_q$.
		\end{enumerate}
	\end{proposition}
	
	\begin{proof}
		
		Before we prove each of the enrichment, let us recall that we have the following comparison between the different classes of maps defining the three model structures. $W_h\subseteq W_m= W_q$, $C_q\subseteq C_m \subseteq C_h$ and $F_h=F_m\subseteq F_q$.
		
		\noindent For (1), let $i:K\hookrightarrow L$ and $j:A\hookrightarrow X$ be respectively a $h$-cofibration and an $q$-cofibration. $i\Square j$ is an $h$-cofibration since $C_q\subseteq C_h$ and $\mathrm{Ch}_{\geq 0}(R)_h$ is monoidal. A similar argument works for the case where $i$ is an acyclic $h$-cofibration. If instead $j$ is an acyclic $q$-cofibration, then by \textbf{Lemma~\ref{q2h-we}}, $j$ is an acyclic $h$-cofibration. Hence, the above argument applies again. 
		
		 \noindent For (2), the argument here is the same as in (i), after replacing \textbf{Lemma~\ref{q2h-we}} by \textbf{Lemma~\ref{m2h-we}}. 
		 
		\noindent For (3), $\mathrm{Ch}_{\geq 0}(R)_m$ is enriched over $\mathrm{Ch}_{\geq 0}(R)_q$, since we have $C_q\subseteq C_m$ and \mbox{$W_m= W_q$}, and also $\mathrm{Ch}_{\geq 0}(R)_m$ is monoidal by \textbf{Remark~\ref{mmonoidal}}. 
	\end{proof}
	
	\begin{proposition}\label{nonqhm}
		For a commutative ring $R$ that admits a non-projective module,
		\begin{enumerate}
			\item $\mathrm{Ch}_{\geq 0}(R)_q$ is not enriched over $\mathrm{Ch}_{\geq 0}(R)_h$,
			\item $\mathrm{Ch}_{\geq 0}(R)_m$ is not enriched over $\mathrm{Ch}_{\geq 0}(R)_h$, 
			\item $\mathrm{Ch}_{\geq 0}(R)_q$ is not enriched over $\mathrm{Ch}_{\geq 0}(R)_m$. 
		\end{enumerate}
	\end{proposition}
	
	\begin{proof}
	For (1), the comparison of the classes of maps gives us $C_q\subseteq C_h$. Consider the $q$-cofibration $i:0\to R[0]$ and the $h$-cofibration $j:0\to L$ that is not a $q$-cofibration, i.e. where $L$ is not degreewise projective. We have that $i\Square j=j:0\to L$ is not a q-cofibration by definition. 
	
	\noindent For (2) \& (3), similar counter-examples as above work since $C_m\subseteq C_h$ and $C_q\subseteq C_m$.
	\end{proof}
	
	\section{Induced Hurewicz model structure on \texorpdfstring{$\mathrm{sMod}_R$}{sModR}}\label{inducedmodel}
	
	Having an equivalence of categories, one can transport a model structure from one category to the other; see the discussion in \cite{transmod} for details. Hence, the Dold--Kan correspondence induces a model structure on the category of simplicial $R$-modules $\mathrm{sMod}_R$ also called the \Def{Hurewicz} model structure and given by:
	\begin{enumerate}
		\item A map $f$ is a weak equivalence if its normalization $N(f)$ is a Hurewicz weak equivalence in $\mathrm{Ch}_{\geq 0}(R)$.
		\item A map $p$ is a fibration if its normalization $N(f)$ is a Hurewicz fibration in $\mathrm{Ch}_{\geq 0}(R)$.
		\item A map $i$ is a cofibration if its normalization $N(f)$ is a Hurewicz cofibration in $\mathrm{Ch}_{\geq 0}(R)$.
	\end{enumerate}
	\noindent Moreover, every object is fibrant and cofibrant.
	\vspace{.5cm}
	
	\textbf{Proposition \ref{EZAWProp}} induces the following.
	
	\begin{proposition}\label{EZAWProp*}
		The chain complex $N(B)^{N(A)}$  is a deformation retract of $N(B^A)$  given by the following two diagrams 
		
		\begin{center}
			\begin{tabular}{cc}
			$\xymatrix{
				N(B)^{N(A)} \ar[r]^-{\mathrm{AW}^*}  \ar@/_3.2pc/[rr]_-{\mathrm{id}}
				& N(B^A) \ar[r]^-{\mathrm{EZ}^*} \ar@{=}[d] & N(B)^{N(A)}, \\
				& & 
			}$
			&
			$\xymatrix{
				N(B^A) \ar[r]^-{\mathrm{EZ}^*}  \ar@/_3.2pc/[rr]_-{\mathrm{id}}
				& N(B)^{N(A)} \ar[r]^-{\mathrm{AW}^*} \ar@{=>}[d] & N(B^A). \\
				& & 
			}$
		\end{tabular}
		\end{center}
		That is the composite $\mathrm{EZ}^* \circ \mathrm{AW}^*$ is the identity, and the composite $\mathrm{AW}^* \circ \mathrm{EZ}^*$ is naturally chain homotopic to the identity. Here,
		\begin{equation*}
			N(B)^{N(A)} := \mathrm{\underline{Hom}}_{\mathrm{Ch}_{\geq 0}(R)}(N(A),N(B))\qquad \text{and}\qquad B^A := \mathrm{\underline{Hom}}_{\mathrm{sMod}_R}(A,B).
		\end{equation*}
	\end{proposition}
	
	\begin{proof} Here we consider the categories  $\mathrm{Ch}_{\geq 0}(R)$ and $\mathrm{sMod}_R$ endowed with their closed symmetric monoidal structures given by \textbf{Lemmas \ref{geq0monoidal}} and \textbf{\ref{sModmonoidal}}. 
		
	\noindent Consider $D^n$ the $n$-disk chain complex, that is the chain complex with $R \xrightarrow{1} R$ in degrees $n$ and $n-1$, and zero in all other degrees. In degree $n$, we have:
		\begin{flalign*}
			& & N(B^A)_n & \cong \mathrm{Hom}_{\mathrm{Ch}_{\geq 0}(R)}(D^n,N(B^A)) & & \\
							& &  & \cong \mathrm{Hom}_{\mathrm{sMod}_{R}}(\Gamma(D^n),B^A) & & \text{by the Dold--Kan correspondence} \\
							 & & & \cong \mathrm{Hom}_{\mathrm{sMod}_{R}}(A\otimes \Gamma(D^n),B) & & \text{by the unenriched tensor-hom} \\
							 & & & & & \qquad\qquad\qquad\qquad\ \text{adjunction}\\
							& &  & \cong \mathrm{Hom}_{\mathrm{Ch}_{\geq 0}(R)}(N(A\otimes \Gamma(D^n)),N(B)) & & \text{by the Dold--Kan correspondence}
		\end{flalign*}
		and
		\begin{flalign*}
			& & \left( N(B)^{N(A)}\right)_n & \cong \mathrm{Hom}_{\mathrm{Ch}_{\geq 0}(R)}(D^n,N(B)^{N(A)}) & & \\
		& &	& \cong \mathrm{Hom}_{\mathrm{Ch}_{\geq 0}(R)}(N(A)\otimes D^n,N(B)) & & \text{by the unenriched tensor-hom}\\
		& & & & & \qquad\qquad\qquad\qquad\ \text{adjunction.}\\
		\end{flalign*}
		Hence, we have the following comparisons of $n$-chains
			\begin{equation*}
				\xymatrix{
					 N(B^A)_n  \ar[d]^-\cong \ar@<1ex>[r]^-{\mathrm{EZ}^*} & \ar@<.1ex>[l]^-{\mathrm{AW}^*} \left( N(B)^{N(A)}\right)_n \ar[d]^-\cong \\
					 \mathrm{Hom}_{\mathrm{Ch}_{\geq 0}(R)}(N(A\otimes \Gamma(D^n)),N(B)) & \mathrm{Hom}_{\mathrm{Ch}_{\geq 0}(R)}(N(A)\otimes D^n,N(B))
				}
			\end{equation*}
			given by:
			\begin{equation*}
				\xymatrix{
				N(B) & \ar[l]_-f \ar@<1ex>[d]^-{\mathrm{AW}} N(A\otimes \Gamma(D^n)) \ar[dr]^-{\mathrm{AW}^*(g)}  & \\
				 & \ar[ul]^-{\mathrm{EZ}^*(f)} N(A)\otimes D^n  \ar[r]_-g \ar@<.1ex>[u]^-{\mathrm{EZ}} & N(B).
				 }
			\end{equation*}
			From \textbf{Proposition \ref{EZAWProp}} and by functoriality, we have $\mathrm{EZ}^*\circ \mathrm{AW}^* = (\mathrm{AW}\circ \mathrm{EZ})^*=\mathrm{id}^*=\mathrm{id}$.
			
			\noindent Also, the map $\mathrm{AW}^* \circ \mathrm{EZ}^*$ is chain homotopic to the identity, by an argument similar to the proof of \cite[Theorem~6.3.1]{re:michael:msc}.
	\end{proof}
	
	\begin{proposition}\label{sModmodel}
		In the Hurewicz model structure on $\mathrm{sMod}_R$, the three classes of maps are characterized as follows.
		\begin{enumerate}
			\item A map $f$ is a weak equivalence if and only if it is a homotopy equivalence.
			\item A map $p$ is a fibration if and only if it satisfies the HLP.
			\item A map $i$ is a cofibration if and only if it satisfies the HEP.
		\end{enumerate}
	\end{proposition}
	
	\begin{proof} Here we consider the categories $\mathrm{Ch}_{\geq 0}(R)$ and $\mathrm{sMod}_R$ enriched, tensored and cotensored over $\mathrm{Ch}_{\geq 0}(\Z)$ and $\mathrm{sAb}$ respectively, given by \textbf{Lemmas \ref{geq0monoidal}} and \textbf{\ref{sModmonoidal}}. 
		
	\noindent (1) Two maps $f,f' : A\to B$ in $\mathrm{sMod}_R$ are homotopic if and only if their normalizations $N(f), N(f') : N(A)\to N(B)$ are chain homotopic \cite[Theorem~8.4.1]{weibel1994introduction}. Therefore $f$ is a homotopy equivalence if and only if $N(f)$ is a chain homotopy equivalence, that is, $f$ is a weak equivalence in $\mathrm{sMod}_R$. 
	
	\noindent (2) $(\implies)$ Let $p:E\to B$ be a fibration in $\mathrm{sMod}_R$. Consider the lifting problem in $\mathrm{sMod}_R$ given by
	\begin{equation*}
		\xymatrix{
			A \ar[d]_-{\iota_0} \ar[r]^-{f} & E \ar[d]^-{p}\\
			A\otimes \Z \Delta^1 \ar[r]_-{g} \ar@{.>}[ur]^-? & B.
		}
	\end{equation*}
	\noindent Applying the normalization $N$,
	\begin{equation*}
		\xymatrix{
			N(A) \ar[d]_-{\iota_0} \ar[dr]^-{N(\iota_0)} \ar[rr]^-{N(f)} & & N(E) \ar[d]^-{N(p)}\\
			N(A)\otimes N(\Z \Delta^1) \ar[r]_-{\mathrm{EZ}} \ar@{.>}[urr]^-{\exists h} & N(A\otimes \Z \Delta^1) \ar[r]_-{N(g)}\ar@{.>}[ur]_-{\exists H} & N(B)
		}
	\end{equation*}
	\noindent where a lift $h$ exists by assumption and the lift $H$, in the square given by the identity $N(g)\circ \mathrm{EZ} = N(p)\circ h$, exists thanks to the fact that the map $\mathrm{EZ}$ is a split monomorphism and a chain homotopy equivalence by \textbf{Proposition \ref{EZAWProp}}, that is $\mathrm{EZ}$ is a trivial Hurewicz cofibration and so satisfies the Left Lifting Property (LLP) with respect to the Hurewicz fibration $N(p)$ in $\mathrm{Ch}_{\geq 0}(R)$. Applying the denormalization $\Gamma$,
	\begin{equation*}
		\xymatrix{
			A \ar[d]_-{\Gamma(\iota_0)} \ar[dr]^-{\iota_0} \ar[rr]^-{f} & & E \ar[d]^-{p}\\
			\Gamma(N(A)\otimes N(\Z \Delta^1)) \ar[r]_-{\Gamma(\mathrm{EZ})} \ar@{.>}[urr]^-{\Gamma(h)} & A\otimes \Z \Delta^1 \ar[r]_-{g}\ar@{.>}[ur]_-{\Gamma(H)} & B.
		}.
	\end{equation*}
	\noindent So the homotopy $\Gamma(H)$ gives a solution to the lifting problem. Therefore $p$ satisfies the HLP in $\mathrm{sMod}_R$. 
	
	\noindent $(\implies)$ Let $p:E\to B$ be a map satisfying the HLP  in $\mathrm{sMod}_R$. Consider the lifting problem in $\mathrm{Ch}_{\geq 0}(R)$ given by
				\begin{equation*}
					\xymatrix{
						A \ar[d]_-{\iota_0} \ar[r]^-{f} & N(E) \ar[d]^-{N(p)}\\
						A\otimes N(\Z \Delta^1) \ar[r]_-{g} \ar@{.>}[ur]^-? & N(B).
					}
				\end{equation*}
				\noindent Applying the denormalization $\Gamma$,
				\begin{equation*}
					\xymatrix{
						\Gamma(A) \ar[d]_-{\iota_0} \ar[rr]^-{\Gamma(f)} & & E \ar[d]^-{p}\\
						\Gamma(A)\otimes \Z \Delta^1 \ar[r]_-{\mu} \ar@{.>}[urr]^-{\exists h} & \Gamma(A\otimes N(\Z \Delta^1)) \ar[r]_-{\Gamma(g)} & B
					}
				\end{equation*}
				\noindent where $\mu: \Gamma(A)\otimes \Z \Delta^1 \to \Gamma(A\otimes N(\Z \Delta^1))$ is the natural transformation that makes $\Gamma$ lax monoidal and a lift $h$ exists by assumption. Applying the normalization $N$,
				\begin{equation*}
					\xymatrix{
						& A \ar@/_/[ddl]_-{\iota_0} \ar[d]^-{N(\iota_0)} \ar[rr]^-{f} & & N(E) \ar[d]^-{N(p)}\\
						& N(\Gamma(A)\otimes \Z \Delta^1) \ar[r]^-{\mathrm{AW}} \ar@{.>}[urr]^-{N(h)} & A\otimes N(\Z \Delta^1) \ar[r]_-{g} & N(B)\\
						A\otimes N(\Z \Delta^1) \ar[ur]^-{\mathrm{EZ}} \ar@{=}[urr]  & & & 
					}
				\end{equation*}
				\noindent where $\mathrm{AW}=N(\mu)$ and $\mathrm{AW}\circ \mathrm{EZ}=\mathrm{id}$ by \textbf{Proposition \ref{EZAWProp}}. So the homotopy \mbox{$H=N(h)\circ \mathrm{EZ}$} gives a solution to the lifting problem. Therefore $p$ is a fibration. 
				
				\noindent (3) $(\implies)$ Let $i:A\to X$ be a cofibration in $\mathrm{sMod}_R$. Consider the lifting problem in $\mathrm{sMod}_R$ given by
				\begin{equation*}
					\xymatrix{
						A \ar[d]_-{i} \ar[r]^-{f} & B^{\Z\Delta^1} \ar[d]^-{ev_0}\\
						X \ar[r]_-{g} \ar@{.>}[ur]^-? & B.
					}
				\end{equation*}
				\noindent Applying the normalization $N$,
				\begin{equation*}
					\xymatrix{
						N(A) \ar[d]_-{N(i)} \ar[r]^-{N(f)} & N(B^{\Z\Delta^1}) \ar[dr]_-{N(ev_0)} \ar[r]^-{\mathrm{EZ}^*} & N(B)^{N(\Z\Delta^1)} \ar[d]^-{ev_0}\\
						N(X) \ar@{.>}[ur]^-{\exists H}\ar@{.>}[urr]^-{\exists h}\ar[rr]_-{N(g)}  &  & N(B)
					}
				\end{equation*}
				\noindent where a lift $h$ exists by assumption and the lift $H$, in the square given by the identity $\mathrm{EZ}^*\circ N(f) = h\circ N(i)$, exists thanks to the fact that the map $\mathrm{EZ}^*$ is a split epimorphism and a chain homotopy equivalence by \textbf{Proposition \ref{EZAWProp*}}, as the image of a split monomorphism and a chain homotopy equivalence by a contravariant functor, that is $\mathrm{EZ}^*$ is a trivial Hurewicz fibration and so satisfies the Right Lifting Property (RLP) with respect to the Hurewicz cofibration $N(i)$ in $\mathrm{Ch}_{\geq 0}(R)$. Applying the denormalization $\Gamma$,
				\begin{equation*}
					\xymatrix{
						A  \ar[d]_-{i} \ar[r]^-{f} & B^{\Z \Delta^1} \ar[dr]_-{ev_0} \ar[r]^-{\Gamma(\mathrm{EZ}^*)} & \Gamma(N(\Gamma(B)^{\Z\Delta^1})) \ar[d]^-{\Gamma(ev_0)} \\
						X \ar@{.>}[ur]^-{\Gamma(H)}  \ar@{.>}[urr]^-{\Gamma(h)} \ar[rr]_-{g} & & B.
					}
				\end{equation*}
				\noindent So the homotopy $\Gamma(H)$ gives a solution to the lifting problem. Therefore $i$ satisfies the HEP in $\mathrm{sMod}_R$. 
				
				 $(\impliedby)$ Let $i:A\to X$ be a map satisfying the HEP  in $\mathrm{sMod}_R$. Consider the lifting problem in $\mathrm{Ch}_{\geq 0}(R)$ given by
				\begin{equation*}
					\xymatrix{
						N(A) \ar[d]_-{N(i)} \ar[r]^-{f} & B^{N(\Z\Delta^1)} \ar[d]^-{ev_0}\\
						N(X) \ar[r]_-{g} \ar@{.>}[ur]^-? & B.
					}
				\end{equation*}
				\noindent Applying the denormalization $\Gamma$,
				\begin{equation*}
					\xymatrix{
						A \ar[d]_-{i} \ar[r]^-{\Gamma(f)} & \Gamma(B^{N(\Z\Delta^1)}) \ar[r]^-{\mu^*} & \Gamma(B)^{\Z\Delta^1} \ar[d]^-{ev_0}\\
						X \ar[rr]_-{\Gamma(g)} \ar@{.>}[urr]^-{\exists h} &  & \Gamma(B)
					}
				\end{equation*}
				\noindent where $\mu^*=\mathrm{Hom}_{\mathrm{Ch}(\Z)}(\mu,N(B))$ and a lift $h$ exists by assumption. Applying the normalization $N$,
				\begin{equation*}
					\xymatrix{
						& & & B^{N(\Z \Delta^1)} \ar@/^/[ddl]^-{ev_0} \\
						N(A)  \ar[d]_-{N(i)} \ar[r]^-{f} & B^{N(\Z \Delta^1)} \ar@{=}[urr] \ar[r]_-{\mathrm{AW}^*} & N(\Gamma(B)^{\Z\Delta^1}) \ar[d]_-{N(ev_0)} \ar[ur]_-{\mathrm{EZ}^*} & \\
						N(X) \ar@{.>}[urr]^-{N(h)} \ar[rr]_-{g} & & B &
					}
				\end{equation*}
				\noindent where $\mathrm{AW}^*=N(\mu^*)$ and $\mathrm{EZ}^*\circ \mathrm{AW}^*=\mathrm{id}$ by \textbf{Proposition \ref{EZAWProp*}}. So the homotopy \mbox{$H=\mathrm{EZ}^*\circ N(h)$} gives a solution to the lifting problem. Therefore $i$ is a cofibration.
	\end{proof}
	
	\begin{proposition}\label{sModmonoidalmodel}\text{\ }
		\begin{enumerate}
			\item For a commutative ring $R$, the category $\mathrm{sMod}_R$ with the Hurewicz model structure is a monoidal model category.
			\item For an arbitrary ring $R$, the Hurewicz model structure on $\mathrm{sMod}_R$ is an enriched model category over the category $\mathrm{sAb}$ of simplicial abelian groups, with the \linebreak
			Hurewicz model structure.
		\end{enumerate}
	\end{proposition}
	
	\begin{proof}
	For item (1), we consider the category $\mathrm{sMod}_R$ endowed with the closed symmetric monoidal structure as given by \textbf{Lemma \ref{sModmonoidal}}. 
	
	\noindent Let $i:X\to Y$ and $k:V\to W$ be cofibrations in $\mathrm{sMod}_R$. Let us show that their pushout product $i\Box k: X\otimes W \amalg_{X\otimes V} Y\otimes V \to Y\otimes W$ is a cofibration too, i.e. for any acyclic fibration $j:A\to B$, the following lifting problem admits a solution
			\begin{equation*}
				\xymatrix{
					X\otimes W \amalg_{X\otimes V} Y\otimes V \ar[d]_-{i\Box k} \ar[r]^-{\widetilde{f}} & A \ar[d]^-{j}\\
					Y\otimes W \ar[r]_-{\widetilde{g}} \ar@{.>}[ur]_-? & B.
				}
			\end{equation*}
			\noindent Applying the normalization $N$,
			\begin{equation}\label{2ndlast}\tiny
				\xymatrix{
					N(X)\otimes N(W) \amalg_{N(X)\otimes N(V)} N(Y)\otimes N(V) \ar[d]_-{N(i)\Box N(k)} \ar[r]^-{\mathrm{EZ}\amalg_{\mathrm{EZ}} \mathrm{EZ}} &  N(X\otimes W \amalg_{X\otimes V} Y\otimes V) \ar[ddl]_-{N(i\Box k)} \ar[r]^-{N(\widetilde{f})} & N(A) \ar[dd]^-{N(j)}\\
					N(Y)\otimes N(W) \ar[d]_-{\mathrm{EZ}} & & \\
					N(Y\otimes W) \ar[rr]_-{N(\widetilde{g})} \ar@{.>}[uurr]^-{\exists h} & & N(B)
				}
			\end{equation}
			\noindent where the lift $h$, in the square given by the identity 
			\begin{equation*}
				N(j)\circ(N(\widetilde{f})\circ(\mathrm{EZ}\amalg_{\mathrm{EZ}}\mathrm{EZ}))=N(\widetilde{g})\circ(\mathrm{EZ}\circ (N(i)\Box N(k))),
			\end{equation*}
			 \noindent exists thanks to the fact that $\mathrm{EZ}\circ (N(i)\Box N(k))$ is the composition of acyclic cofibrations by \textbf{Proposition \ref{EZAWProp}} and so satisfies the LLP with respect to the acyclic fibration $N(j)$ in $\mathrm{Ch}_{\geq 0}(R)$, since the Hurewicz model structure turns $\mathrm{Ch}_{\geq 0}(R)$ into a monoidal model category as we have shown above. Applying the denormalization $\Gamma$,
			\begin{equation*}\tiny
				\xymatrix{
					\Gamma\left( N(X)\otimes N(W) \amalg_{N(X)\otimes N(V)} N(Y)\otimes N(V)\right)  \ar[d]_-{\Gamma(N(i)\Box N(k))} \ar[r]^-{\Gamma\left( \mathrm{EZ}\amalg_{\mathrm{EZ}} \mathrm{EZ}\right) } &  X\otimes W \amalg_{X\otimes V} Y\otimes V \ar[ddl]_-{i\Box k} \ar[r]^-{\widetilde{f}} & A \ar[dd]^-{j}\\
					\Gamma(N(Y)\otimes N(W)) \ar[d]_-{\Gamma(\mathrm{EZ})} & & \\
					Y\otimes W \ar[rr]_-{\widetilde{g}} \ar@{.>}[uurr]^-{ \Gamma(h)} & & B
				}
			\end{equation*}
		Now assume moreover that $i:X\to Y$ and $k:V\to W$ be cofibrations (with either of the two being additionally a weak equivalence) in $\mathrm{sMod}_R$. We have to show that their pushout product $i\Box k: X\otimes W \amalg_{X\otimes V} Y\otimes V \to Y\otimes W$ is an acyclic cofibration. We have that $N(i)\Box N(k)$ is an acyclic cofibration in the monoidal model category $\mathrm{Ch}_{\geq 0}(R)_h$. Hence, considering the upper-left triangle of the diagram $(\ref{2ndlast})$, we have $N(i)\Box N(k)$, $\mathrm{EZ}\amalg_{\mathrm{EZ}} \mathrm{EZ}$ and $\mathrm{EZ}$ are all chain homotopy equivalences in $\mathrm{Ch}_{\geq 0}(R)_h$ \textbf{Proposition \ref{EZAWProp}}. So, by \mbox{2-out-of-3}, $N(i\Box k)$ is a chain homotopy equivalence. Therefore, $i\Box k$ is an acyclic cofibration in $\mathrm{sMod}_R$.
		
		\noindent For item (2), consider the proof of item (1) where we take map $k$ from $\mathrm{sAb}$ and apply item (2) from \textbf{Proposition \ref{monoidalprop}}.
	\end{proof}

\section{Bousfield model structure on \texorpdfstring{$\mathrm{Ch}^{\geq 0}(\cA)$}{Ch0(A)}} \label{Bousfield}

For a bicomplete abelian category $\cA$ with an injective class $\cI$, that is the dual of a projective class \cite[\S2]{christensen1998ideals}, one can recover the Bousfield model structure on the category $\mathrm{Ch}^{\geq 0}(\cA)$ of non-negatively graded cochain complexes over $\cA$. 
For a bicomplete pointed category $\cA$ (not necessarily abelian) with an injective class $\cI$, since the opposite category $(c\cA)^{\mathrm{op}}$ of cosimplicial objects in $\cA$ satisfies $(c\cA)^{\mathrm{op}} \cong s(\cA^{\mathrm{op}})$, applying \mbox{\cite[Theorem~6.3]{christensen2002quillen}} to the opposite category $\cA^{\mathrm{op}}$, we obtain the following model structure on $c\cA$.

\begin{proposition}\label{chris63}
	If for each $X$ in $c\cA$ and each $I$ in $\cI$, $\cA(X,I)$ is a fibrant simplicial set, then the following classes of maps form a model structure on $c\cA$.
	\begin{enumerate}
		\item A map $f$ is a weak equivalence if it is an $\cI$-equivalence, that is $\cA(f,I)$ is a weak equivalence of simplicial sets for each $I$ in $\cI$.
		\item  A map $f$ is a cofibration if $\cA(f,I)$ is a fibration of simplicial sets for each $I$ in $\cI$.
		\item  A map $f$ is a fibration if it has the RLP with respect to all $\cI$-trivial cofibrations.
	\end{enumerate}
\end{proposition}

Assume now that $\cA$ is abelian, then the hypothesis of \textbf{Proposition~\ref{chris63}} holds and so the Dold--Kan correspondence satisfying the diagram
\begin{equation*}
	\xymatrix{
			s(\cA^{\mathrm{op}}) \ar[d]^-\cong \ar@<1.8ex>[r]^-N & \ar@<.6ex>[l]^-{\Gamma}_-\cong \mathrm{Ch}_{\geq 0}(\cA^{\mathrm{op}}) \ar[d]^-\cong \\
			     (c\cA)^{\mathrm{op}} \ar[r]^-\cong     & (\mathrm{Ch}^{\geq 0}(\cA))^{\mathrm{op}}
			     	 }
\end{equation*}
induces a model structure on $\mathrm{Ch}^{\geq 0}(\cA)$ given by the following, that can be found in \mbox{\cite[\S4.4]{bousfield2003cosimplicial}}.

\begin{corollary}[Bousfield model structure]
	Let $\cA$ be a bicomplete abelian category with an injective class $\cI$. The category $\mathrm{Ch}^{\geq 0}(\cA)$ of non-negatively graded cochain complexes of objects in $\cA$ has a model structure given as follows.
	\begin{enumerate}
		\item A map $f$ is a weak equivalence if $\cA(f,I)$ is a quasi-isomorphism in $\mathrm{Ch}(\Z)$ for each $I$ in $\cI$.
		\item  A map $f$ is an cofibration if it is $\cI$-monic in positive degrees (but not necessarily in degree 0) in $\cA$, that is $\cA(f,I)$ is surjective in positive degrees for each $I$ in $\cI$.
		\item  A map $f$ is an fibration if it is a degreewise split  epimorphism with \mbox{$\cI$-injective} kernel.
	\end{enumerate}
	Moreover, every complex is cofibrant, and a complex is fibrant if and only if  it is a complex of $\cI$-injectives.
\end{corollary}

Taking $\cI$ to be the trivial injective class yields the following example.

\begin{example}\label{b-mod}
	For $\cA$ a bicomplete abelian category, $\mathrm{Ch}^{\geq 0}(\cA)$ has a model structure given by:
	\begin{enumerate}
		\item A map $f$ is a weak equivalence if it is a cochain homotopy equivalence.
		\item A map $f$ is a fibration if it is a degreewise split epimorphism.
		\item A map $f$ is a cofibration if it is a degreewise split monomorphism in positive degrees.
	\end{enumerate}
	Moreover, every complex is fibrant and cofibrant.
\end{example}

	\bibliographystyle{alpha}
	\bibliography{MyBib}
	
	\vspace*{1cm}
\end{document}